\documentclass[12pt]{article}
\usepackage{enumerate}
\usepackage{amsmath}
\usepackage{amsthm}
\usepackage{amsfonts}
\usepackage{amssymb}
\usepackage{xcolor}
\usepackage[numbers]{natbib}
\usepackage{graphicx,xcolor}
\theoremstyle{plain}
\newtheorem{question}{Question}

\newtheorem{theorem}[question]{Theorem}
\newtheorem{proposition}[question]{Proposition}

\newtheorem{lemma}[question]{Lemma}

\newtheorem{remark}[question]{Remark}
\theoremstyle{definition}
\newtheorem{definition}[question]{Definition}
\newtheorem*{question*}{Question}

\numberwithin{question}{section}
\numberwithin{equation}{section}
\newtheorem*{theorem*}{Theorem}

\title{Maker-Breaker Percolation Games II: Escaping to Infinity}
\author{A. Nicholas Day\thanks{Institutionen f\"or matematik och matematisk statistik, Ume{\aa} Universitet, 901 87 Ume{\aa}, Sweden. Emails:  \texttt{a.nick.day@gmail.com} and \texttt{victor.falgas-ravry@umu.se}. Research supported by Swedish Research Council grant 2016-03488.} \and Victor Falgas--Ravry\footnotemark[1]}
\begin{document}
\maketitle
\begin{abstract}
Let $\Lambda$ be an infinite connected graph, and let $v_0$ be a vertex of $\Lambda$. We consider the following positional game. Two players, Maker and Breaker, play in alternating turns. Initially all edges of $\Lambda$ are marked as unsafe. On each of her turns, Maker marks $p$ unsafe edges as safe, while on each of his turns Breaker takes $q$ unsafe edges and deletes them from the graph. Breaker wins if at any time in the game the component containing $v_0$ becomes finite. Otherwise if Maker is able to ensure that $v_0$ remains in an infinite component indefinitely, then we say she  has a winning strategy. This game can be thought of as a variant of the celebrated Shannon switching game.  Given $(p,q)$ and $(\Lambda, v_0)$, we would like to know: which of the two players has a winning strategy?

Our main result in this paper establishes that when $\Lambda = \mathbb{Z}^2$ and $v_0$ is any vertex, Maker has a winning strategy whenever $p\geq 2q$, while Breaker has a winning strategy whenever $2p\leq q$. In addition, we completely determine  which of the two players has a winning strategy for every pair $(p,q)$ when $\Lambda$ is an infinite $d$-regular tree. Finally, we give some results for general graphs and lattices and pose some open problems. 
\end{abstract}
\section{Introduction}\label{section: introduction}
Let $\Lambda$ be an infinite connected (multi)graph, and let $v_0$ be a vertex of $\Lambda$. We consider the following positional game, which we call the \emph{$(p,q)$-percolation game on $(\Lambda, v_0)$}. 
\begin{definition}[$(p,q)$-percolation game]\label{definition: percolation game}
Two players, Maker and Breaker, play in alternating turns, with Maker playing first. Initially all edges of $\Lambda$ are marked as unsafe. On each of her turns, Maker marks $p$ unsafe edges as safe, while on each of his turns Breaker takes $q$ unsafe edges and deletes them from the graph. Breaker wins if at any time in the game the component containing $v_0$ becomes finite. Otherwise if Maker is able to ensure that $v_0$ remains in an infinite component indefinitely, then we say she  has a winning strategy. 
\end{definition}
If $\Lambda$ is a vertex-transitive graph, then the choice of $v_0$ does not matter, and we simply speak of the $(p,q)$-percolation game on $\Lambda$.  Our main concern in this paper is to determine for given $(p,q)$ and $(\Lambda, v_0)$ which of the two players has a winning strategy in the corresponding $(p,q)$-percolation game.

\subsection{Main results}\label{subsection: results and organisation of the paper}
The \emph{$d$-dimensional integer lattice} is the graph with vertex set $\mathbb{Z}^d$ whose edges consist of pairs of vertices $\mathbf{v}, \mathbf{w}\in \mathbb{Z}^d$ lying at Euclidean distance $\| \mathbf{v}-\mathbf{w}\|=1$ from each other.  In a standard abuse of notation, we write $\mathbb{Z}^d$ to denote this graph.  Our main results in this paper are the following theorems.

\begin{theorem}\label{theorem: (p,p)-game on integer lattice}
Maker has a winning strategy for the $(1,1)$-percolation game on $\mathbb{Z}^{2}$.  More generally, Maker has a winning strategy for the $(p,p)$-percolation game on $\mathbb{Z}^{d}$ for every integer $1\leq p < d$.
\end{theorem}

\begin{theorem}\label{theorem: (2q,q)-game on integer lattice}
Let $p,q\in \mathbb{N}$.  If $p \geqslant 2q$, then Maker has a winning strategy for the $(p,q)$-percolation game on $\mathbb{Z}^{2}$.
\end{theorem}

\begin{theorem}\label{theorem: (p,2p)-game on integer lattice}
Let $p,q\in \mathbb{N}$.  If $q \geqslant 2p$, then Breaker has a winning strategy for the $(p,q)$-percolation game on $\mathbb{Z}^{2}$.
\end{theorem}

\begin{theorem}\label{theorem: regular trees}
Let $p,q,d\in \mathbb{N}$, and let $T_{d}$ denote the infinite $d$-regular tree.  Maker has a winning strategy for the $(p,q)$-percolation game on $T_{d}$ if and only if $p(d-2) \geqslant q$.
\end{theorem}
	\begin{definition}\label{def: bi-regular trees}
		For integers $2\leq a\leq b$, let $T_{a,b}$ denote the infinite bi-regular tree with vertices of two types, Type I and Type II, defined as follows:  each vertex of Type I is adjacent to exactly $a$ vertices of Type II, while each vertex of Type II is adjacent to exactly $b$ vertices of Type I.  
	\end{definition}
	\begin{theorem}\label{theorem: bi-regular tree}
		Let $p,q,a,b \in \mathbb{N}$ with $2\leq a\leq b$.  Maker has a winning strategy in the $(p,q)$-percolation game on $T_{a,b}$ if and only if
		\[p(b-2) -\Bigl\lceil\frac{p}{a}\Bigr\rceil (b-a)\geq q,\]
		 irrespective of the choice of the root.
	\end{theorem}

\subsection{Background and discussion}\label{subsection: background and motivation}
The $(p,q)$-percolation game can be thought of as a variant of the celebrated \emph{Shannon switching game}, where the identity of the winner under optimal play was determined by Lehman~\cite{Lehman64}.  The Shannon switching game is played on a finite connected graph $G$ with two pre-specified vertices $u,v$. Two players, Short and Cut, play in alternating turns, with Cut playing first. Initially all edges of $G$ are marked as unsafe. In each of her turns Short selects an unsafe edge and marks it as safe, while in each of his turns Cut selects an unsafe edge and deletes it. Short wins if she managed to create a path of safe edges from $u$ to $v$, otherwise Cut wins. 

The possibility of selecting multiple edges radically changes the nature of the game, and in particular Lehman's arguments for the Shannon switching game do not appear to carry over to the $(p,q)$ setting. Mention should also be made here of the game of \emph{Gale}, or \emph{Bridg-it}, a specific class of Shannon switching games where $G$ is a rectangular grid and Short seeks to construct a path of safe edges from the left-hand side to the right-hand side. Bridg-it was popularised by Martin Gardner~\cite{Gardner58} and made into a commercially available game sold by Hasbro.  Several winning strategies for Bridg-it are known, but again they do not seem to generalise to the $(p,q)$ case.

As the names given to the two players indicate, $(p,q)$-percolation games are an instance of the more general class of Maker--Breaker games. Such games are played on a board $X$ (i.e. a set --- in our case, the edge-set of $\Lambda$), and a collection $\mathcal{W}$ of subsets of $X$ called \emph{winning sets}. Two players, Maker and Breaker, take turns to claim elements of $X$. Maker (typically) plays first, and claims $a$ elements in each of her turns, while Breaker claims $b$ elements on each of his. Maker wins if she manages to claim all the elements from some winning set $W\in \mathcal{W}$, while Breaker wins if he thwarts her by claiming at least one element from each winning set. Since the board is finite, no draws are allowed, and the main question is to determine who has a winning strategy.

Strictly speaking, in the $(p,q)$-percolation game one could argue that Breaker rather than Maker is trying to claim all elements from some winning set --- namely he seeks to claim all edges from a cut-set in $\Lambda$ that disconnects $\Lambda$ in such a way that the origin lies in a finite component.  However, there are two reasons to name the players in our percolation games as we do.  The first is that, informally speaking, we think of Maker as trying to build a `path to infinity' from the origin, and would like to view these infinite paths as her winning sets.  The second reason is that in Theorem~\ref{theorem: compactness} in Section~\ref{section: compactness}, we show that Maker has a winning strategy for the $(p,q)$-percolation game on $\Lambda$ if and only if she has winning Maker strategies for certain collections of Maker--Breaker games played on finite subsets of $\Lambda$ (in which her winning sets are paths from $v_0$ to some target set of vertices). It thus seems apt to call her Maker as we do.%In these games she truly is trying to `make' a winning set, and so it seems apt to call her Maker.

Maker--Breaker games on graphs have been extensively investigated  since the foundational work of Chv\'atal and Erd{\H o}s~\cite{ChvatalErdos78} in the late 1970s. Important examples of such games include the connectivity game, the $k$-clique game and the Hamiltonicity game, where the board $X$ consists of the edges of a complete graph on $n$ vertices and the winning sets are spanning trees, $k$-cliques and Hamiltonian cycles respectively.

Chv\'atal and Erd{\H o}s proved that, for a variety of such games, if $n$ is sufficiently large, then Maker has a winning strategy in the case where $a=b=1$. In each case they then asked how large a bias $b=b(n)$ was required for the $(1,b)$ versions of these games to turn into Breaker's win and provided a surprising and influential \emph{random graph heuristic} for determining the value of these \emph{threshold biases}. Namely, according to this heuristic the threshold bias $b_{\star}$ at which Breaker has a winning strategy should lie  close to the threshold $b$ for a set of $\frac{1}{b+1}\binom{n}{2}$ edges chosen uniformly at random to fail, with high probability, to contain any winning set. 

This random graph heuristic has been widely investigated by a large number of researchers, in particular by Beck~\cite{Beck82,Beck85, Beck93, Beck94} and Bednarska and {\L}uczak~\cite{BednarskaLuczak00, BednarskaLuczak01}. Its correctness has been rigorously established for some games, such as the connectivity~\cite{GebauerSzabo09}, $k$-clique~\cite{Beck08} and Hamiltonicity~\cite{Krivelevich11} games, but it has also been shown to fail for other games such as general $H$-games~\cite{BednarskaLuczak00} (where the winning sets are copies of some fixed, finite graph $H$ containing at least three non-isolated vertices).

In a different direction, Stojakovi{\'c} and Szab{\'o}~\cite{StojakovicSzabo05} considered playing these Maker--Breaker games on random boards, by having $X$ consist of the edges of an Erd{\H o}s--R\'enyi random graph $G_{n,\theta}$. As having fewer edges cannot help Maker, the natural question in this setting is: what is the threshold $\theta_{\star}$ such that if $\theta\gg \theta_{\star}$, then with probability $1-o(1)$ Maker has a winning strategy for the $(1,1)$-game on $G_{n,\theta}$, while if $\theta\ll \theta_{\star}$, then with probability $1-o(1)$ Breaker has a winning strategy?  Stojakovi{\'c} and Szab{\'o} showed that for some games, such as the connectivity games,  $1/b_{\star}$ and $\theta_{\star}$ are of the same order, but that for others, such as the triangle game, no such relationship holds.

These subtle and intriguing connections between Maker--Breaker games and resilience phenomena  in discrete random structures (in addition to their obvious combinatorial appeal) have led to an abiding interest in Maker--Breaker games. %We refer a reader to the 2008 monograph of Beck~\cite{Beck08} for a summary and exposition of some of the many results in the area known at that point.
Our main motivation in this paper is to investigate whether any connections similar to the Erd{\H o}s--Chv\'atal random graph heuristic exist when we move over from the world of finite graphs to the world of percolation on infinite graphs. Percolation theory is the branch of probability theory concerned, broadly speaking, with the study of random subgraphs of infinite graphs, and in particular with the almost-sure emergence of infinite connected components. Since its inception in Oxford in the late 1950s, it has blossomed into a rich and active area of mathematical research.  One of the cornerstones of the theory is the celebrated Harris--Kesten Theorem~\cite{Harris60, Kesten80} which we state below. The $\theta$-random measure $\mu_\theta$ on an infinite graph $\Lambda$ is, informally, the probability measure on subsets of $E(\Lambda)$ that includes each edge with probability $\theta$, independently of all the others. (We eschew measure-theoretic subtleties here; for a rigorous definition of $\mu_{\theta}$ via cylinder events, see Bollob\'as and Riordan~\cite[Chapter 1]{BollobasRiordan06}.) 
\begin{theorem*}[Harris--Kesten Theorem]
	Let $G_\theta$ denote a $\mu_\theta$-random subgraph of $\mathbb{Z}^2$. Then
	\begin{itemize}
		\item if $\theta \leqslant \frac{1}{2}$, then almost surely $G_\theta$ does not contain an infinite component;
		\item if $\theta >\frac{1}{2}$, then almost surely $G_\theta$ contains an infinite component.
	\end{itemize}
\end{theorem*}
The motivation for the present paper is the following question, inspired by the Harris--Kesten theorem and the work of Chv\'atal and Erd{\H o}s.
\begin{question*}%\label{question: HK for Maker-Breaker}
Consider the $(p,q)$-percolation game on $\mathbb{Z}^2$. Does there exist a critical bias $b_{\star}> 0$ such that for any $\varepsilon>0$ and all $p$ sufficiently large,  $(b_{\star}+\varepsilon)p<q$ implies Breaker has a winning strategy, while $(b_{\star}-\varepsilon) p>q$ implies Maker has a winning strategy?
\end{question*}
\noindent Our Theorems~\ref{theorem: (2q,q)-game on integer lattice} and~\ref{theorem: (p,2p)-game on integer lattice} show that there exist constants
\begin{align*}
b^- = \sup\{b: \ \forall (p,q) \textrm{ with }bp \geq q, \textrm{Maker has a winning strategy for the $(p,q)$-percolation game on $\mathbb{Z}^2$}\},\\
b^+ = \inf\{b: \ \forall (p,q) \textrm{ with }bp \leq q, \textrm{Breaker has a winning strategy for the $(p,q)$-percolation game on $\mathbb{Z}^2$}\}
\end{align*}
with 
\[\frac{1}{2}\leq b^- \leq b^+ \leq 2.\]
Furthermore, Theorem~\ref{theorem: (p,p)-game on integer lattice} shows that $b^+\geq 1$. Based on the Harris--Kesten theorem and a random graph heuristic, it would be tempting to guess that $b^-=b^+=b_{\star}=1$ --- however, we currently have far too little evidence in favour of this guess to make it a formal conjecture.

On the other hand, we are able to establish some correspondence between percolation theory and our Maker-Breaker games in the case of infinite $d$-regular trees, also known as \emph{Bethe lattices} in the context of percolation theory. It is an easy exercise to determine that the critical probability for percolation on an infinite $d$-regular tree is $\theta_c=1/(d-1)$.  Our Theorem~\ref{theorem: regular trees}  establishes that if $(d-2)p\geq q$ then Maker has a winning strategy for the $(p,q)$-percolation game on the $d$-regular tree, while if $(d-2)p<q$ Breaker has a winning strategy. Thus, in this case we do have a critical bias $b_{\star}=d-2$, and it matches what one might expect from a random graph heuristic, that is, $\theta_{c} = 1/(b_{\star} + 1)$.

We also show the existence of a critical bias $b_{\star}=b-2-\frac{b-a}{a}=\frac{ab-a-b}{a}$ for $(p,q)$-percolation games on the bi-regular trees $T_{a,b}$. In this case our result does not quite match what one might expect from a random graph heuristic, since $1/(b_{\star}+1)=a/(b(a-1))$, whereas the critical probability for percolation is (via standard application of the theory of branching processes)  $\theta_{c}=1/\sqrt{(a-1)(b-1)}$.

%Crucial tool in modern proofs: RSW lemmas. Investigated corresponding games (which are related to HEX, Bridg'it and the Shannon switching game) in our first paper. In present paper use our results and additional work to investigate the percolation games.

Many of the modern proofs of the Harris--Kesten theorem rely on Russo--Seymour--Welsh (RSW) lemmas on the probability of having a crossing path in a $\theta$-random subset of a rectangular grid. Before considering Question~\ref{question: HK for Maker-Breaker} and the $(p,q)$-percolation, it thus makes sense to first investigate certain crossing games --- in effect, generalisations of Bridg-it. This was done by the authors in a prequel~\cite{DayFalgasRavry18a} to this work. In the present paper we apply some  results on crossing games from~\cite{DayFalgasRavry18a}  together with some other ideas to obtain Theorems~\ref{theorem: (2q,q)-game on integer lattice} and~\ref{theorem: (p,2p)-game on integer lattice}.

Finally, as pointed out to us by one of the referees,  the $(p,q)$-percolation game on $\mathbb{Z}^2$ is related to John Horton Conway's celebrated \emph{Angel and Devil game}. In this two-player game, an Angel starts at the origin in $\mathbb{Z}^2$, under the watchful gaze of her opponent, the Devil. The Angel plays first. On each of her turns, the Angel can jump to any vertex of $\mathbb{Z}^2$ at $\ell_{\infty}$-distance at most $p$ from her present location (leaping over any obstacle), provided this vertex has not yet been destroyed by the Devil; on each of his  turns the Devil destroys a vertex of $\mathbb{Z}^2$. Here $p$ is a fixed positive integer, referred to as the Angel's power. The Devil wins if the Angel is unable to move on her turn. The Angel wins if she survives indefinitely.

The Angel and Devil game was introduced by Berlekamp, Conway and Guy~\cite{BerlekampConwayGuy82} in 1982 and later popularised by Conway~\cite{Conway96}. It was shown in~\cite{BerlekampConwayGuy82} that if $p=1$ then the Devil has a winning strategy. Conway asked~\cite{Conway96} if for some sufficiently large power $p$, the Angel has a winning strategy. Conway's question (and the rewards he offered for its resolution) sparked the interest of researchers. Sufficiently powerful Angels were shown to win a version of the game on $\mathbb{Z}^3$~\cite{BollobasLeader06,Kutz05}, before the independent work of Bowditch~\cite{Bowditch07},  G\'acs~\cite{Gacs07}, Kloster~\cite{Kloster07} and Math\'e~\cite{Mathe07} showed that Angels of sufficiently large power $p$ win the Angel and Devil game on $\mathbb{Z}^2$, thereby answering Conway's question.  Indeed, Math\'e and Kloster gave Angel strategies showing that a power of $p=2$ suffices for the Angel to evade the Devil indefinitely.

Besides being played on similar boards and both having one player (Maker/the Angel) trying to `escape to infinity' while the other (Breaker/the Devil) destroys parts of the board, one of the points in common between $(p,q)$-percolation games and the Angel and Devil Game is the fact that, for many combinations of $(p,q)$, some of the  `obvious' Maker strategies seem to quickly run into trouble, with Breaker able to lay distant traps for her just as the Devil does for the Angel in Conway's game.

\subsection{Organisation of the paper}
In Section~\ref{section: compactness}, we prove a compactness result, Theorem~\ref{theorem: compactness}, which reduces a percolation game on an infinite connected graph to certain families of Maker--Breaker games on finite connected (multi)graphs. In Section~\ref{section: results on general graphs and lattices} we prove a result for general graphs as well as Theorems~\ref{theorem: (p,p)-game on integer lattice}, \ref{theorem: regular trees} and~\ref{theorem: bi-regular tree} on percolation games on integer lattices and infinite trees. In Section~\ref{section: results on the square integer lattice}, we prove our results for percolation games on $\mathbb{Z}^2$, that is, Theorems \ref{theorem: (2q,q)-game on integer lattice} and \ref{theorem: (p,2p)-game on integer lattice}. We end the paper in Section~\ref{section: concluding remarks and open questions} with some concluding remarks and many questions and open problems.

\subsection{Notation}\label{subsection: basic def and notation}

A (multi)graph is a pair $G = (V, E)$, where $V = V (G)$ is a set of vertices and $E = E(G)$ is a collection of pairs from $V$ (some of which may be included multiple times), which form the edges of $G$. Given $S\subseteq V(G)$, the subgraph of $G$ induced by $S$ is $(S, \{e\in E(G): \ e\subseteq S\})$.
%A subgraph $H$ of $G$ is a graph with $V (H) \subseteq V (G)$ and $E(H) \subseteq E(G)$; $H$ is said to be an induced subgraph if $E(H)=\{e\in E(G): \ e\subseteq V(H)\}$.
%Given $n \in \mathbb{N}$, let $[n] = \{1, 2, \ldots, n\}$. 
In this paper, we  often omit the prefix (multi) and identify a graph with its edge-set when the underlying vertex-set is clear from context. A graph $G$ is connected if for any pair of vertices $v,v'\in V(G)$ there is a finite path of edges of $G$ from $v$ to $v'$. We study percolation games on \emph{rooted infinite connected (multi)graphs}, or \emph{RIC}-graphs, which is to say pairs $(\Lambda, v_0)$ where $\Lambda$ is an infinite, connected graph and the root $v_0$ is a vertex of $\Lambda$. We often refer to $v_0$ as the \emph{origin} of our RIC-graph. In the case where $\Lambda$ is vertex-transitive, we may in a slight abuse of notation write $\Lambda$ for the RIC-graph $(\Lambda, v_0)$ and take it that $v_0$ is arbitrarily specified.

The $(p,q)$-percolation game on the RIC-graph $(\Lambda, v_0)$ was formally defined above in Definition~\ref{definition: percolation game}. When analysing it, it is convenient to consider that both players choose the edges they will select sequentially, so e.g.\ on each of her turns Maker selects an ordered sequence of $p$ edges, $e_1, e_2, \ldots, e_p$, she marks as safe one by one (rather than a $p$-set of edges which she marks safe all at once) and similarly Breaker selects an ordered sequence of $q$ edges to destroy one by one. This clearly makes no difference to the game, and allows us to define a notion of time:  we say a game is at \textit{time} $t$ if a combined total of $t$ edges have been played.  For example, at time $t = 9$ in a $(2,2)$-percolation game, it is Maker's third turn and she has already claimed the first (but not the second) of the two edges she will mark as safe in that turn.

\section{Compactness}\label{section: compactness}
In this section we prove a compactness theorem on percolation games, which  says that they are essentially equivalent to certain generalised Shannon switching games. Let $(\Lambda, v_0)$ be an RIC-graph, and let $X$ be a finite connected induced subgraph of $\Lambda$ such that $v_0 \in V(X)$. The \emph{boundary} of $X$ in $\Lambda$ is the collection $\partial_{\Lambda} (X)$ of vertices in $X$ that send an edge to a vertex in $V(\Lambda)\setminus V(X)$. 
\begin{definition}[Escape game]
	In the \emph{$(p,q)$-escape game} on $(X, \Lambda, v_0)$, two players, Maker and Breaker, claim edges of the finite graph $X$ in alternating turns, with Maker playing first.  On each of her turns, Maker claims $p$ as-yet-unclaimed edges for herself, while Breaker on each of his turns claims $q$ such edges for himself. Maker wins the game if she manages to claim all of the edges of a path from the $v_0$ to $\partial_{\Lambda}(X)$, while Breaker  wins if he claims at least one edge from every such path.	
\end{definition}
When there is no ambiguity regarding $p,q, \Lambda, v_0$ we just refer to the escape game on $X$. Observe that the escape game cannot end in a draw, and terminates within $\Bigl\lceil \vert E(X)\vert /(p+q)\Bigr\rceil$ turns. By contracting all vertices in $\partial_{\Lambda}(X)$ to a single vertex $u_0$, we obtain a (multi)graph $G$. The escape game on $X$ is thus seen to be equivalent to a generalised Shannon switching game on $G$ where Short seeks to create a path of safe edges from $u_0$ to $v_0$ and where Short and Cut mark respectively $p$ and $q$ edges on each of their turns.

%Recall that a graph is \emph{locally finite} if every vertex in the graph is incident with only a finite number of edges.
\begin{theorem}\label{theorem: compactness}
Let $\Lambda$ be an RIC-graph with origin $v_{0}$.  Maker has a winning strategy for the $(p,q)$-percolation game on $\Lambda$ if and only if Maker has winning strategies for every $(p,q)$-escape game on $(X, \Lambda, v_0)$, for all finite connected induced subgraphs $X$ of $\Lambda$ with $v_{0}\in V(X)$.
\end{theorem}

\begin{proof}
One direction of the statement is straightforward. Suppose there exists some finite connected induced subgraph $X$ of $\Lambda$ with $v_{0}\in V(X)$ such that Breaker has a winning strategy for the $(p,q)$-escape game on $X$.  Then the following is a winning Breaker strategy for the $(p,q)$-percolation game on $\Lambda$: ignore the rest of the board and focus exclusively on $X$. Whenever Maker, on one of her turns, claims $p'\leq p$ edges of $X$, Breaker selects $p-p'$ as-yet-unclaimed edges of $X$ arbitrarily and pretends that Maker claimed these edges too.  Breaker then responds using the winning strategy for the $(p,q)$-escape game on $X$. Within at most $\Bigl\lceil \vert E(X)\vert /(p+q)\Bigr\rceil$ turns, Breaker will have claimed (and destroyed) at least one edge from every path from $v_0$ to $\partial_{\Lambda}(X)$. In particular by this time, $v_0$ is contained in a finite component which is a strict subset of $V(X)$ and Breaker has won the percolation game on $\Lambda$.

In the other direction, suppose that Maker has a winning strategy for the $(p,q)$-escape game on $X$ for all finite connected induced subgraphs $X$ of $\Lambda$ with $v_{0}\in V(X)$. Let $(X_{i})_{i \geqslant 0}$ be an arbitrary sequence of finite connected induced subgraphs of $\Lambda$ such that $v_0 \in V(X_0)$, $X_i \subseteq X_{i+1}$ for all $i\geq 0$ and $\bigcup_{i\geq 0}X_i= \Lambda$. (For example, one could take each $X_{i}$ to be the subgraph induced by the vertices of $\Lambda$ that lie at graph distance at most $i$ from $v_{0}$ in $\Lambda$.)

By our assumption, at the beginning of Maker's first turn each of the boards $X_i$ is in a winning position for her in the $(p,q)$-escape game on that $X_i$. We shall construct a Maker strategy that ensures this remains true for all subsequent turns, i.e. that at the beginning of her every turn Maker has a winning strategy for the $(p,q)$-escape game on $(X'_i, \Lambda', v_0)$ for every $i\geq 0$, where $X'_i$ and $\Lambda'$ are the (multi)graphs obtained from $X_i$ and $\Lambda$ by deleting all edges claimed by Breaker and contracting all edges claimed by Maker up to that point.

We first show that if we are able to achieve the above, then this is a winning strategy for Maker in the $(p,q)$-percolation game on $\Lambda$. Suppose this is not the case, and that at some time $T$ in the game the component $C$ containing $v_0$ becomes finite.  Let $D$ be the set of vertices in $V(\Lambda)\setminus C$ that are adjacent (by any edge of $\Lambda$ regardless of who may have claimed it) to at least one vertex in $C$.  Then every edge between $C$ and $D$ must have been claimed by Breaker.  Moreover, as Breaker has certainly claimed no more than $T$ edges we have that $D$ is finite.  Thus there exists some $i$ such that $C \cup D \subseteq X_{i}$.  We must have that $C \cap \partial(X_{i}) = \emptyset$, as otherwise there would be some vertex in $D \cap (V(\Lambda)\setminus X_{i})$, which is not possible.  Therefore every path from $v_{0}$ to $\partial(X_{i})$ meets at least one Breaker edge, which contradicts the fact that $X_{i}$ was in a winning position for Maker in the escape game on $X_i$.

We now specify our strategy. Assume that at the beginning of her turn, Maker is in a winning position in the $(p,q)$-escape game on $X_i$ for every $i\geq 0$ (taking into account the edges of $X_i$ claimed by Maker and Breaker during past turns). Then for each $i\geq 0$ there is a winning Maker strategy for the $(p,q)$-game on $X_i$ from the current position specifying some set $P_{i}$ of $p$ as-yet-unclaimed edges of $X_i$ as Maker's next move in the escape game.  We use the sequence of sets $(P_{i})_{i\geqslant 0}$ to produce a set $A$ of $\vert A \vert  \leqslant p$ edges, using the following algorithm.  Set $A_{0} = \emptyset$.  Given a set of edges $A_{j}$, let
\begin{equation}
S_{i,j}  = \begin{cases} P_{i} \setminus A_{j}  \text{ if } A_{j} \subseteq P_{i}, \\  \nonumber 
\emptyset \hspace{0.2cm} \text{ if } A_{j} \not\subseteq P_{i}. \end{cases} \nonumber
\end{equation}
If any edge appears in infinitely many members of the sequence $(S_{i,j})_{i \geqslant 0}$, pick\footnote{If one assumes $V(\Lambda)=\mathbb{N}$, one can e.g. pick the least such edge in the lexicographic order and thereby avoid using the axiom of choice.} one such edge $e$ and set $A_{j+1} = A_{j} \cup \{e\}$.  Otherwise, set $A = A_{j}$ and terminate the process.  As $\vert A_j\vert =j$ and $\vert S_{i,j}\vert\leq p-j$, the algorithm will terminate after at most $p$ iterations, outputting a set $A$ of at most $p$ edges. Maker's move for the percolation game is then to claim all edges in $A$ together with a set $A'$ of $p-\vert A\vert $ arbitrary edges. Breaker then answers by claiming some set $B$  of $q$ edges. We claim that irrespective of the choices of $A'$ and $B$, Maker is still in a winning position in the $(p,q)$-escape game on $X_i$ for every $i\geq 0$ at the beginning of her next turn.

Indeed, suppose for contradiction that Maker is in a losing position on $X_{i_0}$, for some $i_0 \in \mathbb{N}$, i.e. that Breaker has a winning strategy for the $(p,q)$-escape game in that position with Maker playing first.  By construction of $A$, there exists some $i_{1} > i_{0}$ such that $P_{i_1}\cap X_{i_0}=A \cap X_{i_0}$.  Now, as Maker is in a losing position on $X_{i_0}$, we have that Breaker has some strategy that can guarantee that Maker never builds a path from $v_{0}$ to $\partial(X_{i_0})$.  Thus, when playing the $(p,q)$-escape game on $X_{i_{1}}$, Breaker can use this same strategy to ensure that Maker never builds a path from $v_{0}$ to $\partial(X_{i_1})$.  However, this contradicts the fact that $P_{i_{1}}$ was a winning Maker move for the $(p,q)$-escape game on $X_{i_1}$.  Therefore no such $X_{i_{0}}$ can exist, and so every $X_{i}$ is in a winning position for Maker in the corresponding $(p,q)$-escape game.

\end{proof}
%\begin{remark}
%	In the proof above, it is sufficient for Maker to have winning strategies for the $(p,q)$-escape games on a sequence of nested finite connected induced subgraphs $(X_i)_{i\geq 0}$ with $v_0 \in V(X_i)$ and $\bigcup_i  X_i= \Lambda$ --- by the argument we give, this immediately implies that Maker has a winning strategy on the $(p,q)$-escape game on \emph{all} finite connected induced subgraphs $X$ of $\Lambda$ with $v_0\in V(X)$.
%\end{remark}
\begin{remark}\label{remark: (1,1) percolation game is solved}
Lehman~\cite{Lehman64} showed that Short has a winning strategy in the Shannon switching game on a multigraph $G$ with distinguished vertices $u,v$ if and only there is an induced subgraph $H$ of $G$ containing two edge-disjoint spanning trees connecting $u$ to $v$. Together with Theorem~\ref{theorem: compactness}, this in principle gives a criterion for determining which of Maker and Breaker has a winning strategy in the $(1,1)$-percolation game on $\Lambda$ for any RIC-graph $\Lambda$: for each $i\in \mathbb{N}$ consider the subgraph $X_i$ of $\Lambda$ induced by vertices at graph distance at most $i$ from the root $v_0$, contract the boundary $\partial_{\Lambda}(X_i)$ to a single point $u_0$ to obtain a graph $G_i$ and then use Lehman's result to determine which of Short and Cut has a winning strategy for the Shannon switching game played on $G_i$ with distinguished vertices $u_0$ and $v_0$. 
\end{remark}

\section{General graphs, integer lattices and regular trees}\label{section: results on general graphs and lattices}
\subsection{Path colourability and a proof of Theorem~\ref{theorem: (p,p)-game on integer lattice}}
We say that a $k$-colouring of the edges of an RIC-graph $\Lambda$ is a $k$\textit{-path-colouring} if, for every vertex $v$ of $\Lambda$ and every colour $i$, there exists an infinite path through $v$ consisting only of edges of $\Lambda$ in colour $i$.  If such a colouring exists, we say that $\Lambda$ is $k$\textit{-path-colourable}; $k$-path colourability can be viewed as a generalisation of the necessary and sufficient condition for Short to have a winning strategy in the Shannon switching game, see~\cite{Mansfield01}.

\begin{theorem}\label{theorem: coloured lattices}
Let $\Lambda$ be a $(k+1)$-path-colourable RIC-graph. Then Maker has a winning strategy for the $(p,p)$-percolation game on $\Lambda$ for every $p\leq k$ (irrespective of the choice of the root). \\
What is more, Maker can forgo her first turn and in addition ensure that every vertex in $\Lambda$ remains in an infinite component at all times in the game.
\end{theorem}
\begin{proof}
Let $c$ be a $(k+1)$-path-colouring of $\Lambda$ with colours from $[k+1]=\{1, 2,\ldots k+1\}$.  Given finite, pairwise-disjoint edge-sets $X,Y\subseteq \Lambda$, let $(\Lambda-X)/Y$ denote the graph obtained from $\Lambda$ by first deleting every edge in $X$ and then contracting every edge in $Y$. 
\begin{lemma}\label{lemma: good (1,1) response in (k+1)path colouring}
	Let $c$ be a $(k+1)$-path-colouring of an RIC-graph $\Lambda'$ with colours from $[k+1]$. Then for every edge $e\in \Lambda'$ and every colour $j\in [k+1]\setminus \{c(e)\}$, there exists an edge $f\in \Lambda'$ with $c(f)=j$ such that $(\Lambda'-\{e\})/\{f\}$ is $(k+1)$-path colourable. 
\end{lemma}
\begin{proof}
%We begin by showing that, for any edge $e$ in $\Lambda$ and for any colour $j \in [k+1]$ with $j \neq c(e)$, there exists an edge $f = f(\Lambda,e,j)$ with $c(f) = j$ such that the following holds: if $(\Lambda-e)/f$ denotes the graph obtained from $\Lambda$ by deleting the edge $e$ and then contracting the edge $f$, then the $(k+1)$-colouring $(\Lambda-e)/f$  inherits from $c(\Lambda)$ is a $(k+1)$-path-colouring.  In particular, $(\Lambda-e)/f$ is also $(k+1)$-path-colourable.
%Given $e\in \Lambda$ and $j \in [k+1]\setminus\{c(e)\}$, l
Let $\Gamma=\{e'\in E(\Lambda):\ c(e')=c(e)\}$. 
%Moreover, let $\Gamma^{'}$ be the graph obtained by deleting $e$ from $\Gamma$.  
 As  $c$ is a $(k+1)$-path-colouring of $\Lambda$, every component of $\Gamma$ is infinite. In particular,  $\Gamma-\{e\}$ contains at most one finite component.  If $\Gamma-\{e\}$ contains no finite component, then let $f$ be an arbitrary edge of $\Lambda$ with $c(f)=j$.  Otherwise, $\Gamma-\{e\}$ has a unique finite component $D$.  Let $v$ be any vertex in $D$.  As $c$ is a $(k+1)$-path-colouring of $\Lambda$, there is an infinite path $P$ in $\Lambda$ through $v$ all of whose edges are assigned colour $j\neq c(e)$ by $c$.  As $P$ is infinite and $D$ is finite, there exists an edge $f = \{v_{1},v_{2}\} \in P$ such that $v_{1} \in D$, $v_{2} \notin D$.  In particular there exists a finite path $P_1$ from $v$ to $v_1$ and an infinite path $P_2$ through $v_2$ such that both $P_1$ and $P_2$ are monochromatic  with colour $c(e)$ (with respect to the colouring $c$).  Thus, when we delete $e$ from our $(k+1)$-coloured graph $\Lambda$ and contract the edge $f$ to obtain $(\Lambda-\{e\})/\{f\}$, we have that every vertex $v\in D$ is part of an infinite component in colour $c(e)$ --- and hence the same holds for every vertex $v\in V(\Lambda)$. This immediately implies that  our $(k+1)$-coloured graph $(\Lambda-\{e\})/\{f\}$ is $(k+1)$-path-colourable (the property that every vertex is part of an infinite monochromatic component in colour $i$ for every $i\in [k+1]\setminus\{c(e)\}$ is inherited from $\Lambda$).
\end{proof}
%We are now ready to describe Maker's strategy in the $(p,p)$-percolation game on $\Lambda$.    %For ease of notation, we will assume that Maker forgoes her first turn, so that the game begins with Breaker to play first (in reality Maker just plays her first $p$ edges arbitrarily). 
For each $s \in \mathbb{N}$, let $M_{s}$ (respectively $B_s$) be the set of edges claimed by Maker (respectively Breaker) on her (respectively his) $s$-th turn of the game.  Set $\Lambda^{0} = \Lambda$, and for each $s \in \mathbb{N}$ let $\Lambda^{s} = (\Lambda^{s-1}-B_{s})/M_{s+1}$.  We construct below a strategy for Maker that ensures $\Lambda^{s}$ is $(k+1)$-path-colourable for every $s\geq 0$.  This is a winning strategy in the strongest sense possible, as it guarantees that every vertex $v\in V(\Lambda)$ (not just $v_0$) remains in an infinite component at all times throughout the game.

By assumption, $\Lambda^{0}$ is $(k+1)$-path-colourable. On her first turn, Maker claims $p$ arbitrary edges, or forgoes her turn entirely (this makes no difference to our analysis). On subsequent turns, she responds to Breaker's moves as follows.  Suppose $\Lambda^{s-1}$ has a $(k+1)$-path-colouring $c$ with colours from $[k+1]$, and that in his $s$-th turn Breaker claims the edges $B_{s} = \{b_{1},\ldots,b_{p}\}$.  As $\vert B_{s}\vert  = p  \leqslant k$, there is some colour $j\in [k+1]$ which does not appear in $\{c(b_1), c(b_2), \ldots , c(b_p)\}$. 
%We now use Lemma~\ref{lemma: good (1,1) response in (k+1)path colouring} to devise Maker's response. 
Set $c_0=c$, $\Lambda^{s-1}_{0} = \Lambda^{s-1}$, and define sequences of $(k+1)$-path-colourings $c_i$ and graphs $\Lambda^{s-1}_{i}$ as follows.  For $1\leq i\leq p$, assume $c_{i-1}$ is a $(k+1)$-path-colouring of the graph $\Lambda^{s-1}_{i-1}$ with colours from $[k+1]$.  By Lemma~\ref{lemma: good (1,1) response in (k+1)path colouring}, given an edge $b_i$ with $c_{i-1}(b_i)\neq j$  there exists an edge $f_i$ in colour $c_{i-1}(f_i)=j$ such that the graph $\Lambda^{s}_{i}=\left(\Lambda^s_{i-1} - \{b_i\}\right)/\{f_i\}$ is $(k+1)$-path-colourable by the colouring $c_i$ it inherits from $\Lambda^s_{i-1}$. It follows that the graph $\Lambda^{s-1}_p= \left(\Lambda^{s-1}-B_{s}\right)/\{f_1, f_2, \ldots, f_p\}$ is $(k+1)$-path-colourable. Setting $M_{s+1}=\{f_1, f_2, \ldots, f_p\}$ to be the set of edges claimed by Maker on her $(s+1)$-th turn, we have that $\Lambda^{s}$ is $(k+1)$-path-colourable. By induction on $s$, this gives our desired winning strategy for Maker.
\end{proof}
\noindent Theorem~\ref{theorem: (p,p)-game on integer lattice} now follows as an easy corollary:
%\begin{corollary}[
%Maker has a winning strategy for the $(p,p)$-percolation game on $\mathbb{Z}^{d}$ for all $p < d$.
%\end{corollary}
\begin{proof}[Proof of Theorem~\ref{theorem: (p,p)-game on integer lattice}]
Label the axes of $\mathbb{Z}^{d}$ as $x_{1}, x_2, \ldots, x_{d}$.  For each $i \in [d]$, colour all edges parallel to the $x_{i}$-axis with colour $i$.  This is a $d$-path-colouring of $\mathbb{Z}^{d}$. By Theorem \ref{theorem: coloured lattices}, it follows that Maker has a winning strategy for the $(p,p)$-percolation game on $\mathbb{Z}^{d}$ for all integers $p$: $1\leq p < d$.
\end{proof}

\subsection{Infinite trees}

We begin this section by proving a general lemma on percolation games played on trees.  Let $T$ be any RIC-tree with root $v_{0}$.  Given two edges $e_{1}, e_{2}$  in $T$, we say that $e_{1}$ is an \textit{ancestor} of $e_{2}$ if the unique path between $e_{2}$ and $v_{0}$ contains the edge $e_{1}$.  Conversely if $e_{1}$ is an ancestor of $e_{2}$, then we say that $e_{2}$ is a \textit{descendant} of $e_{1}$.  Note that every edge is both an ancestor and a descendant of itself.

	Let $C_{t}$ be the component of Maker edges that, at time $t$, includes $v_{0}$.  Throughout this subsection, whenever we refer to $C_{t}$ we will always take it to mean the component containing $v_{0}$ at the relevant point in time, and forgo writing ``at time $t$".
%We first will show that we may always assume Maker and Breaker always play their edges adjacent to $C_{t}$.  We note that 
\begin{lemma}\label{lemma: tree percolation}
We may assume that under optimal play in the  $(p,q)$-percolation game on $T$, Maker and Breaker only claim edges adjacent to $C_{t}$.
\end{lemma}
\begin{proof}
We begin by showing this holds for Breaker. Suppose Breaker's strategy requires him to claim an edge $e$ that is not adjacent to $C_{t}$.  Let $e'$ be the ancestor of $e$ that is closest to $v_{0}$ but is not in $C_{t}$; this edge $e'$ has not yet been claimed by Maker as it does not lie in $C_t$. Furthermore, as far as the outcome of the percolation game is concerned,  Breaker claiming $e'$ is equivalent to Breaker simultaneously claiming both $e'$ and all its descendants (a set of edges that includes $e$). Thus claiming $e'$ is at least as good for Breaker as claiming $e$, and we may thus assume without loss of generality that Breaker does claim $e'$ rather than $e$ when following an optimal strategy.

We next show that we may assume that at time $t$ Maker only claims edges adjacent to $C_{t}$.  Suppose that $S$ is a winning strategy for Maker in the $(p,q)$-percolation game on $T$, and at some point in the game $S$ requires Maker to claim an edge $e$ that is not adjacent to $C_{t}$.  Once again, let $e'$ be the ancestor of $e$ that is closest to $v_{0}$ but is not in $C_{t}$.  If $e'$ has already been claimed by Breaker, then as far as the outcome of the game is concerned this is equivalent to Breaker having claimed every descendant of $e'$. In particular Maker playing $e$ has no bearing on the outcome of the game, and we may thus assume without loss of generality  that Maker claims some other edge adjacent to $C_t$ rather than $e$ when following a winning strategy.

On the other hand, suppose that Breaker has not claimed $e'$. Then we construct a new Maker strategy $S'$, in which Maker claims $e'$ instead of $e$ at time $t$, but then follows what the strategy $S$ would dictate if she had claimed $e$ (i.e. she plays $e'$ but pretends she played $e$ and keeps following $S$ accordingly). Three things can happen.
\begin{enumerate}[(i)]
	\item If $S$ at some future time $t' > t$ requires Maker to claim $e'$, then our new strategy $S'$ will require her to pick $e$ at that time instead. From then on $S$ and $S'$ are identical winning strategies for Maker. 
	\item If at some future time $t' > t$ Breaker claims the edge $e$, then Maker is in no worse a position than if she had been following the strategy $S$ and Breaker had claimed the edge $e'$ at time $t'$ instead.  This is because, once again, Breaker claiming the edge $e'$ is equivalent to them claiming $e'$ and all its descendants (a set of edges that includes $e$).  Thus Maker may pretend that Breaker picked $e'$ at time $t'$ (and continue to pretend that she picked $e$ at time $t$). From then on, $S$ and $S'$ are identical winning strategies for Maker. 
	\item  If neither of the above ever occurs, then $S$ and $S'$ are indistinguishable winning strategies for Maker from time $t+1$ onwards.
\end{enumerate} 
We have thus constructed from $S$ a new winning strategy $S'$ for Maker in which she only picks edges adjacent to $C_t$ at time $t$. This concludes the proof of the Lemma.
\end{proof}
Recall that $T_{d}$ denotes the rooted $d$-regular infinite tree.  We are now ready to prove Theorem \ref{theorem: regular trees}, which stated that Maker wins the $(p,q)$-percolation game on $T_{d}$ if and only if $p(d-2) \geqslant q$.

\begin{proof}[Proof of Theorem \ref{theorem: regular trees}]
As before, let $C_t$ be the component of Maker edges that contains the root at time $t$.  By Lemma \ref{lemma: tree percolation} we may assume that both Maker and Breaker always play their edges adjacent to $C_{t}$.  For each $t \in \mathbb{N}$, let $B_{t}$ be the set of edges that Breaker has claimed at time $t$, and let $\delta_{t}$ be the number of unclaimed edges adjacent to $C_{t}$ at time $t$.  Breaker wins the percolation game if and only if there exists some $t \in \mathbb{N}$ such that $\delta_{t} = 0$.  Note that $\delta_{0} = d$.  Suppose $\delta_t>0$ at some time $t$; if it is Maker's turn to claim an edge, then $\delta_{t + 1} = \delta_{t} + d-2$, while if it is Breaker's turn then $\delta_{t + 1} = \delta_{t} -  1$.  In particular if Maker has not lost before the end of Breaker's $N$-th turn, for $N\geq 0$, then 
\[\delta_{N(p + q)} = \delta_{0} + N\left(p(d-2) - q\right)= d +N\left(p(d-2) - q\right)\]
and it is Maker's turn to play. Thus, if $p(d-2) - q \geqslant 0$, then $\delta_{t} \geq d>0$ for all $t \in \mathbb{N}$, and the game goes on indefinitely, which is a win for Maker.  On the other hand, if $p(d-2) - q < 0$ then the game lasts at most $\big\lceil \frac{d}{q - p(d-2)} \big\rceil$ rounds before $\delta_t$ hits $0$ and Breaker wins the game.
\end{proof}

With a little more work, we can use Lemma~\ref{lemma: tree percolation} to prove Theorem~\ref{theorem: bi-regular tree} for percolation games on bi-regular trees. Recall the definition of the infinite bi-regular tree $T_{a,b}$ from the introduction, Definition~\ref{def: bi-regular trees}. For $\epsilon\in \{$I, II$\} $ we denote by  $T^{\epsilon}_{a,b}$ the rooted  tree obtained from $T_{a,b}$ by selecting an arbitary vertex of Type $\epsilon$ to be the root.
\begin{proof}[Proof of Theorem~\ref{theorem: bi-regular tree}]
		Let $T=T^{\epsilon}_{a,b}$. As before, let $C_t$ denote the component of Maker edges containing the origin at time $t$. By Lemma~\ref{lemma: tree percolation}, we may assume without loss of generality that at each time $t$, both players only claim edges adjacent to $C_t$. Let $X_t$ and $Y_t$ denote the collections of vertices in $V(T)\setminus C_t$ of Type I and II respectively that are adjacent, by an unclaimed edge, to a vertex of $C_t$ at time $t$.

		The effects of Maker's and Breaker's moves on the vector $(\vert X_t\vert , \vert Y_t\vert )$ are easy to describe. If Maker claims an edge between $C_{t}$ and $X_{t}$, then $(\vert X_{t+1}\vert , \vert Y_{t+1}\vert )=(\vert X_t\vert -1, \vert Y_t\vert + a-1)$, while if she claims an edge between $C_{t}$ and $Y_{t}$, then $(\vert X_{t+1}\vert , \vert Y_{t+1}\vert )=(\vert X_t\vert + b-1, \vert Y_t\vert -1)$.  Similarly, if Breaker claims an edge between $C_{t}$ and $X_{t}$, then $(\vert X_{t+1}\vert , \vert Y_{t+1}\vert )=(\vert X_t\vert -1, \vert Y_t\vert )$, while if he claims an edge between $C_{t}$ and $Y_{t}$, then $(\vert X_{t+1}\vert , \vert Y_{t+1}\vert )=(\vert X_t\vert, \vert Y_t\vert -1)$.

		Write $d_N$ for the $\ell_1$-distance from $(\vert X_t\vert , \vert Y_t\vert )$ to $(0,0)$ at the beginning of Maker's $(N+1)$-th turn, i.e. $d_N:= \vert X_{N(p+q)}\vert +\vert Y_{N(p+q)}\vert$.
		Note that  Breaker wins the game if and only if he can ensure $d_N=0$ for some $N\in \mathbb{Z}_{\geq0}$. Finally, set $r_{\star}=\bigl\lceil\frac{p}{a}\bigr\rceil $, and write $\Delta= p(b-2) -r_{\star}(b-a)- q$.

Our claim is that Maker has a winning strategy if and only if $\Delta\geq 0$. Before we dive into the proof, let us give a heuristic for this. Both players in the $(p,q)$-percolation game on $T$ have an obvious \emph{greedy strategy}, namely: claim any edge whose addition to $C_t$ would result in the largest increase in the edge-boundary (in other words, pick an edge from $C_t$ to $Y_t$ if $Y_t\neq \emptyset$, otherwise pick an edge from $C_t$ to $X_t$).  As we show in Lemma~\ref{lemma: distance change over a turn} below, in the long run the quantity $\Delta$ essentially bounds the change in the $\ell_1$-distance $d_N$ between successive turns under these strategies. In particular, if $\Delta$ is nonnegative then Maker can ensure  that $d_N>0$ for all $N\geq 0$, while if $\Delta<0$ then Breaker can over time reduce $d_N$ to $0$ and win the game. For Maker, this is essentially the whole story. For Breaker's winning strategy, there are some subtleties: Breaker must first reduce $\vert Y_t\vert $ to zero before he can reduce the distance $d_N$. Further, if $b=a+1$, then he cannot ensure that $\vert Y_t\vert$ remain zero at the start of every Maker turn, and even showing that he can ensure $\vert Y_t\vert=0$ at the end of \emph{some} turn requires a little care.

	Having outlined the proof, we now give the details. We begin by dealing with a trivial case, By definition, we know $a\leq b$. If $a=b$ then in fact we can merge the two types into just one: $T$ is then just a copy of $T_{a}$, the rooted $a$-regular infinite tree. By Theorem~\ref{theorem: regular trees}, we thus have that Maker has a winning strategy in the $(p,q)$-percolation game if and only if $q\leq p(b-2)= p(b-2)- r_{\star}(b-a)$, as desired. So from now on, let us assume $a<b$. 
	\begin{lemma}\label{lemma: distance change over a turn}
		Suppose $d_N>0$.Then the following hold:	
		\begin{enumerate}[(i)]
			\item by using her greedy strategy, Maker can ensure that $d_{N+1}\geq d_N + \Delta$;
			\item if $Y_{N(p+q)}=\emptyset$, then by using his greedy strategy, Breaker can ensure that $d_{N+1}\leq d_N +\Delta $.
		\end{enumerate}
	\end{lemma}
	\begin{proof}\ \newline
		\noindent \textbf{Part  (i):}	proving the first part of this lemma is simple counting. If $Y_{N(p+q)}=\emptyset$, then Maker's strategy dictates that she first claims an edge from $C_t$ to a Type I vertex (such an edge exists since $d_N>0$) and then edges from $C_t$ to Type II vertices whenever possible. In particular, she claims exactly $r_{\star}$ edges to Type I vertices and $(p-r_{\star})$ edges to Type II vertices. Otherwise if $Y_{N(p+q)}\neq \emptyset$ Maker can get away with claiming only $r$ edges to Type I vertices, for some $r\leq r_{\star}$. Either way, Maker's moves increase the $\ell_1$-distance from $(\vert X_t\vert, \vert Y_t\vert )$ to $(0,0)$ by at least 
		\begin{equation}
		r_{\star}(a-2)+(p-r_{\star})(b-2)=\Delta +q. \nonumber
		\end{equation}
Breaker's move then reduce the distance again by $q$, whence at the start of Maker's next turn $d_{N+1}\geq d_N +\Delta$.
		\newline
		\noindent \textbf{Part  (ii):} as noted above, if  $Y_{N(p+q)}=\emptyset$ then during her $(N+1)$-th turn Maker must claim $r$ edges from $C_t$ to $X_t$ for some $r$:  $r_{\star}\leq r\leq p$. Thus $\vert X_t\vert + \vert Y_t \vert $ increases by at most $\Delta+q$ as a result of Maker's moves. By following his greedy strategy, Breaker then ensures that $d_{N+1}\leq d_N +\Delta$. %Further, Maker's moves add $c=r(a_2-1)+ (p-r)(b_2-1)\leq \max\Bigl(r_{\star}a_2+(p-r_{\star})(b_2-1),pa_2\Bigr)$ vertices to $Y_{N(p+q)}$, for some $r$: $r_{\star}\leq r\leq p$. Breaker's greedy strategy then removes $\min(c,q)$ of those vertices, yielding the desired bound on $\vert Y_{(N+1)(p+q)}\vert$, 
	\end{proof}	
Lemma~\ref{lemma: distance change over a turn} immediately implies that Maker has a winning strategy for the $(p,q)$-percolation game on $T$ whenever $\Delta\geq 0$. Thus all that remains to show is that Breaker has a winning strategy if $\Delta <0$. Since $p,q,a,b$ are integers and since increasing Breaker's power can only help him, it suffices to show that Breaker has a winning strategy if $\Delta=-1$, i.e. if 
\begin{align}\label{eq: q assumption}
q= p(b-2)-r_{\star}(b-a) +1
\end{align}We have two cases to consider.
\newline
\noindent\textbf{Case 1: $b\geq a+2$.} In this case, observe that in any turn, Maker's moves can increase $Y_t$ by at most $p(a-1)$. Since we have $a\geq 2$, $p\geq 1$ and
\begin{eqnarray}
q-p(a-1)&=& (p-r_{\star})(b-a) -p+1\geq 2\left(p-\frac{p+a-1}{a}\right)-p+1 \nonumber \\
&=& \frac{(a-2)(p-1)}{a}  \geq 0,
\end{eqnarray}
Breaker's greedy strategy ensures that the number of edges to Type II vertices available to Maker at the start of her turn is non-increasing. Moreover in any turn where the number of such edges has not strictly decreased, the number of edges to Type I vertices available to Maker must have gone down by at least $p$. Thus within at most $\vert X_0\vert/p$ turns, $\vert Y_{N(p+q)}\vert $ must have decreased by  at least $1$. What is more, at the beginning of the first turn $N$ on which $\vert Y_{N(p+q)}\vert <\vert Y_{0}\vert $ we have $\vert X_{N(p+q)}\vert \leq \vert X_0\vert +p(b-1)$ (since $\vert X_{(N-1)(p+q)}\vert \leq \vert X_0\vert$ and since Maker's moves during her $N$-th turn can increase $\vert X_t\vert $ by at most $p(b-1)$). This readily implies that within at most
\[\frac{1}{p}\Bigl(\vert X_0\vert + \left(\vert X_0\vert +p(b-1)\right)+ \left(\vert X_0\vert +2p(b-1) \right)+\cdots +\left(\vert X_0\vert +(\vert Y_0\vert-1) p(b-1)\right) \Bigr)\leq (b-1)(d_0)^2\]
turns, $Y_t$ is empty at the beginning of Maker's turn. Suppose this occurs at the start of turn $N_0$, i.e. $Y_{N_0(p+q)}=\emptyset$. Then Breaker's greedy strategy further ensures that $Y_{(N_0+1)(p+q)}=\emptyset$ and, by Lemma~\ref{lemma: distance change over a turn} part (ii), that $d_{N_0+1}\leq d_{N_0} +\Delta= d_{N_0} -1$. Within a further $N_1= d_{N_0}$ turns, Breaker will therefore win the game. Breaker's greedy strategy thus ensures his victory in the $(p,q)$-percolation game within at most a total of $N_0+N_1\leq N_0+ (p (b-2)-q) N_0\leq \left(p(b-2)-q+1 \right)(b-1)(d_0)^2$ turns, where the bound on $N_1$ is from the trivial upper bound on the change in the $\ell_1$-distance $d_N$ between consecutive Maker turns. So even if Maker were given a head start in the game (by being allowed to claim some edges before the game starts, thus making $d_{0}$ large), Breaker would still win.
\newline
\noindent\textbf{Case 2: $b= a+1$.} We are analysing the game on the tree $T_{a,a+1}$. Substituting the value of $b=a+1$ into our assumption~\eqref{eq: q assumption} on the value of $q$, we obtain 
\begin{align}\label{eq: q assumption mod}
q-p(a-1)=-(r_{\star}-1).
\end{align}
In particular, Breaker cannot prevent Maker from increasing the number of edges to Type II vertices between turns, and more care is needed to show his greedy strategy is successful.

Suppose that the game starts from some position $(X_0, Y_0)$, and that during the first $N$ turns, Breaker is able to claim only edges to Type II vertices. Let $\theta Np$ denote the number of edges to Type I vertices claimed by Maker during her first $N$ turns. Then the following hold:
\begin{eqnarray}
\vert X_{N(p+q)}\vert & =& \vert X_0\vert  + Np \left(-\theta +a(1-\theta)\right),
 \label{eq: x evolution}\\
 \vert Y_{N(p+q)}\vert & = & \vert Y_0\vert  + Np \left((a-1)\theta -(1-\theta)-\frac{q}{p}\right).\label{eq: y evolution}
\end{eqnarray}
(Recall that in this case we have $b=a+1$, so the $a$ in equation~\eqref{eq: x evolution} is equal to $b-1$.) Since $\vert X_{N(p+q)}\vert, \vert Y_{N(p+q)}\vert \geq 0$, inequalities \eqref{eq: x evolution} and \eqref{eq: y evolution}  together imply
\begin{align}\label{eq: theta bounds}
- \frac{\vert Y_0\vert}{aNp} + \frac{p+q}{ap}\leq \theta \leq \frac{a}{a+1} + \frac{\vert X_0\vert}{(a+1)Np}.
\end{align}
Now by~\eqref{eq: q assumption mod}, $q=p(a-1)-r_{\star}+1$. Substituting this value into the difference $\frac{p+q}{pa} -\frac{a}{a+1}$ and using the bound $r_{\star}=\left\lceil \frac{p}{a}\right\rceil< \frac{p}{a}+1$, we get
\begin{align*}
\frac{(p+q)}{pa}-\frac{a}{a+1}= \left(\frac{pa-r_{\star}+1}{pa}\right) -1 + \frac{1}{a+1} &= 1- \frac{(r_{\star}-1)}{pa} -1+\frac{1}{a+1}\\&> -\frac{(p/a)}{pa} +\frac{1}{a+1}=\frac{1}{a+1}-\frac{1}{a^2}>0,
\end{align*}
for all $a\geq 2$.
Together with~\eqref{eq: theta bounds}, this implies that there exists some finite constant $N_0$ depending only on $a$, $p$ and $d_0=\vert X_0\vert +\vert Y_0\vert$ such that at the latest on his turn $N_0$ Breaker is forced to claim at least one edge to a vertex of Type I. Explicitly, the difference between the left-hand side of~\eqref{eq: theta bounds} and the right-hand side is at least
\begin{align*}
\frac{p+q}{pa}-\frac{a}{a+1}-\frac{d_0}{aNp}>\frac{1}{a+1}-\frac{1}{a^2}-\frac{d_0}{aNp},
\end{align*}
which is strictly positive for all $N$ satisfying \begin{align}\label{eq: bound on phase lenght bireg trees}
N\geq N_0= \Bigl\lceil \frac{d_0}{ap}{\left(\frac{1}{a+1}-\frac{1}{a^2}\right)}^{-1}\Bigr\rceil=\Bigl\lceil \frac{a d_0}{p(a^2-a-1)}\Bigr\rceil.\end{align}
This contradicts~\eqref{eq: theta bounds}, and shows that the assumption that Breaker could claim only edges to Type II vertices must have failed by his turn $N_0$ at the latest.

We now divide the game into phases. At the beginning of each phase it is Maker's turn to play, and $Y_t$ is empty. Throughout a phase, Breaker follows his greedy strategy and claims only edges to Type II vertices; a phase then ends with the first turn in which Breaker claims an edge to a vertex of Type I. For completeness, we say Phase $0$ begins with Maker's first move and ends at the end of the first turn in which Breaker claims an edge to a Type I vertex.

We have just shown above that Breaker's greedy strategy ensures every phase is finite. We now claim that for all $k\geq1$, the $\ell_1$-distance $d_{N}$ at the start of Phase $(k+1)$ is strictly smaller than it was at the start of Phase $k$.  Observe that this will immediately imply that Breaker's greedy strategy wins the game for Breaker within some finite number of turns, thereby completing the proof of Theorem~\ref{theorem: bi-regular tree}.

For convenience, let us shift the game-time so that Phase $k$ begins at time $0$, with thus $d_0 =\vert X_0\vert$ and $\vert Y_0\vert=0$, and let us assume the Phase $k$ ends at the $N$-th turn. Let $\theta Np$ denote the number of edges to Type I vertices claimed by Maker during Phase $k$. Then we have
 \begin{align}\label{eq: x evolution redux}
 d_N=\vert X_{N(p+q)}\vert & = d_0 + Np \left(-\theta +a(1-\theta)\right) - q',
 \end{align}
 \begin{align}\label{eq: y evolution redux}
0= \vert Y_{N(p+q)}\vert & = Np \left((a-1)\theta -(1-\theta)-\frac{q}{p}\right)+q',
 \end{align}
where $q'\in [q]$ is the number of edges to Type I vertices Breaker claims in his last turn of the phase. 
%As in \eqref{eq: theta bounds}, 
From~\eqref{eq: y evolution redux} we can derive a lower bound 
\begin{align}\label{eq: theta lower bound redux}
\theta = -\frac{q'}{aNp} + \frac{p+q}{pa}\geq -\frac{q}{ap}+\frac{p+q}{ap}=\frac{1}{a}.
\end{align}
Adding \eqref{eq: x evolution redux}  and \eqref{eq: y evolution redux}, and substituting in our lower bound~\eqref{eq: theta lower bound redux} for $\theta$, we obtain
 \begin{align*}
 d_N=d_N+0= d_0 + Np\left(a-1 -\theta-\frac{q}{p} \right)&\leq d_0+N\left(p(a-1)-\frac{p}{a}-q\right)\\
 & = d_0 +N\left(p(a-1)- \frac{p}{a}- p(a-1)+(r_{\star}-1)\right)<d_0,
 \end{align*}
 where in the last line we used equality~\eqref{eq: q assumption mod} to eliminate $q$ and the bound $r_{\star}-1=\lceil\frac{p}{a}\rceil-1< \frac{p}{a}$ to get the strict inequality. Thus $d_N<d_0$, as claimed, and Breaker's greedy strategy is a winning one. Appealing to \eqref{eq: bound on phase lenght bireg trees}, we can moreover give an upper bound on the number of turns before Breaker's victory, even if Maker is given a head start.

 Indeed, suppose Maker was allowed to claim some connected set of edges before the start of the game, so that at the beginning of Breaker's very first turn, $\vert X_t\vert +\vert Y_t\vert =d$.  Then by~\eqref{eq: bound on phase lenght bireg trees}, Phase $0$ lasts at most $N_1=\left\lceil \frac{ad}{p(a^2-a-1)}\right\rceil<\frac{8d}{pa}$ Breaker turns (here we use the bound $a^2-a-1\geq a^2/4$ for $a\geq 2$). At the beginning of Phase $1$, $\vert X_t\vert +\vert Y_t\vert$ is then at most $d'= N_1((a-1)p-q)<8d$. There are then at most $d'$ Phases played before Breaker wins, each of which lasts at most $N_2=\left\lceil\frac{ad'}{p(a^2-a-1)}\right\rceil<\frac{64d}{pa}$ Breaker turns. Thus Breaker wins after playing at most
\[N_1+ N_2 d' < \frac{8d}{pa}+8d\cdot\frac{64d}{pa}<\frac{1000 d^2}{pa}\] of his turns.
\end{proof}	
\noindent Let us remark here that the proof of Theorem~\ref{theorem: bi-regular tree} essentially boiled down to the analysis of the following vector game.
\begin{definition}[Vector game]
Let $a',b',p,q,x_0, y_0$ be positive integers. Two players, Maker and Breaker, play in alternating turns, with Maker playing first. At the onset of the game, the players are given a play vector $(X, Y)=(x_0,y_0)$ in the positive quadrant of $\mathbb{Z}^2$, which they take turns at modifying. 

On each of her turns, Maker is allowed to perform $p$ moves, where a Maker move consists of changing the play vector $(X, Y)$ by adding $(-1, a')$ or $(b',-1)$ to it, subject to the restriction that the play vector must remain inside the positive quadrant of $\mathbb{Z}^2$. On each of his turns, Breaker is allowed to perform $q$ moves, where a Breaker move consists of changing the play vector $(X,Y)$ by adding $(-1,0)$ or $(0,-1)$ to it, again subject to the restriction that the play vector remains inside the positive quadrant of $\mathbb{Z}^2$. Breaker wins the game if the play vector ever takes the value $(0,0)$, while Maker is said to have a winning strategy if she can prevent this from happening.
\end{definition}
This vector game is an instance of the broader class of \emph{energy games} and \emph{vector addition systems with states (VASS)}, which are widely studied within computer science, albeit primarily from a complexity perspective (see e.g.~\cite{BrazdilJancarKucera10, FahrenbergJuhlLarsenSrba11}). Such games are motivated by the problem of a manufacturer who uses raw materials to produce various items. The manufacturer can buy batches of raw materials, while customers buy items from the manufacturer. Customer demand is fickle, perhaps even adversarial, and the manufacturer's aim is to maintain sufficient levels of raw materials and cash to run his operation indefinitely. This is clearly reminiscent of Maker's aim to avoid $(X, Y)=(0,0)$ in the vector game above, and points to a possible application of our Maker-Breaker games and especially of our winning Maker strategies.

\section{The square integer lattice}\label{section: results on the square integer lattice}
Our goal in this section is to prove Theorems~\ref{theorem: (2q,q)-game on integer lattice} and~\ref{theorem: (p,2p)-game on integer lattice} on percolation games on the square integer lattice $\mathbb{Z}^2$. Since $\mathbb{Z}^2$ is a planar graph (with its natural embedding in $\mathbb{R}^2$), we can define its \emph{planar dual} $(\mathbb{Z}^2)^{\star}$, which is the graph with a vertex for each face of $\mathbb{Z}^2$ and an edge $e^{\star}$ for each edge $e$ of $\mathbb{Z}^2$, where $e^{\star}$ joins the two vertices of $(\mathbb{Z}^2)^{\star}$ corresponding to the two faces of $\mathbb{Z}^2$ in whose boundary $e$ lies; we say that such edges $e$, $e^{\star}$ form a dual pair.

A key property of the square integer lattice is that it is \emph{self-dual}: $(\mathbb{Z}^2)^{\star}$ is isomorphic to $\mathbb{Z}^2$. It is customary to represent $(\mathbb{Z}^2)^{\star}$ in the plane as a copy of $\mathbb{Z}^2$ shifted by $(1/2, 1/2)$, so that its vertex set is $\mathbb{Z}^2+(1/2, 1/2)$ and the straight line segments corresponding to a dual pair of edges $(e,e^{\star})$ intersect in their midpoint.  If $e$ is a horizontal edge in $\mathbb{Z}^2$, that is $e = \{(x,y),(x+1,y)\}$ for some $x,y\in \mathbb{Z}$, then we denote both $e$ and its dual $e^{\star}$ by their midpoint, so that $e =  (x+0.5,y)$ and $e^{*} = (x+0.5,y)^{*}$.   Similarly, for a vertical edge $e = \{(x,y),(x,y+1)\}$ for some $x,y\in \mathbb{Z}$, we denote both $e$ and its dual $e^{\star}$ by their midpoint, so that $e =  (x,y+0.5)$ and $e^{*} =  (x,y+0.5)^{*}$. This identification of edges/dual edges with their midpoints will allow us to view our Maker--Breaker games as being played on two boards: the two players compete to claim midpoints, with one player claiming them to obtain edges from the board $\mathbb{Z}^2$, while the other does so to obtain edges from the dual board $\left(\mathbb{Z}^2\right)^{\star}$.

A key tool used in this section will be a result from~\cite{DayFalgasRavry18a} on the following auxiliary game. %an auxiliary game we call \textit{$q$-double-response game}.
\begin{definition}[$q$-double-response game]
	Let $\mathbb{Z}\times P_n$ be the subgraph of $\mathbb{Z}^{2}$ induced by the vertex set $\{(x,y):x \in \mathbb{Z}, y \in [n]\}$ and let $\left(\mathbb{Z}\times P_n\right)^{\star}$ be the graph of edges dual to the edges in $\mathbb{Z}\times P_n$.  The $q$\textit{-double-response game} is a positional game played by a horizontal player $\mathcal{H}$ and a vertical player $\mathcal{V}$. The game begins with $\mathcal{V}$ playing first. On each of his turns, $\mathcal{V}$ first picks an integer $r \in [q]$ and then claims $r$ as-yet unclaimed edges in $\mathbb{Z}\times P_n$ for himself;  $\mathcal{H}$ then responds by claiming $2r$  as-yet unclaimed edges for herself in response to $\mathcal{V}$'s move.

	The vertical player $\mathcal{V}$ wins if he is able to claim a set of edges $D$ whose dual $D^{\star}$ is a top-bottom crossing path in $\left(\mathbb{Z}\times P_n\right)^{\star}$ (equivalently, the edges in $D$ are a  vertical cut through $\mathbb{Z}\times P_n$).  The horizontal player $\mathcal{H}$ wins if she is able to indefinitely prevent $\mathcal{V}$ from building such a dual path.
	
	%In this game, $\mathcal{V}$'s aim is to claim a set of edges corresponding to a top-bottom crossing path of dual edges, and we say $\mathcal{V}$ wins if he is able to do so.  The horizontal player $\mathcal{H}$'s aim is to prevent this from ever happening, and we say $\mathcal{H}$ wins the game if she is able to do so.  
\end{definition}
%In order to prove Theorems \ref{theorem: (2q,q)-game on integer lattice} and \ref{theorem: (p,2p)-game on integer lattice}, we will make use of the following theorem from \cite{DayFalgasRavry18a}:
\begin{proposition}\label{proposition: double-response}[\cite[Theorem~4.1]{DayFalgasRavry18a}]
If $n \geqslant q+1$, then $\mathcal{H}$ has a winning strategy for the $q$-double-response game on $\mathbb{Z}\times P_n$.
\end{proposition}
\begin{remark}\label{remark: horizontal crossing path}
If $\mathcal{H}$ has a winning strategy for the $q$-double-response game on $\mathbb{Z}\times P_n$, then for any $m\geq 0$ she also has a winning strategy for the $q$-double-response game restricted to the subgraph $P_m \times P_n$ induced by the vertices $(x,y)$: $x\in [m], y\in [n]$. By~\cite[Lemma 2.1]{DayFalgasRavry18a}, this implies she has a strategy for building a left-right crossing path through the rectangle $P_m \times P_n$.
\end{remark}
	%We remark that throughout this section we will always view the game through the lens of duality, so that $\mathcal{V}$ always claims dual edges.
Theorem~\ref{theorem: (2q,q)-game on integer lattice} follows as an easy consequence of Proposition~\ref{proposition: double-response}.
\begin{proof}[Proof of Theorem \ref{theorem: (2q,q)-game on integer lattice}]
On her  first turn, Maker claims all $p$ edges of the vertical path $P$ from the origin $(0,0)$ to $(0,p)$.  From then on, Breaker can win the game only if he manages to claim a set of edges $D$ such that the corresponding set of dual edges $D^{\star}$ contains a cycle  in $\left(\mathbb{Z}^2\right)^{\star}$ with $P$ in its interior; in particular,  $D^{\star} \cap \left(\mathbb{Z}\times P_p\right)^{\star}$ would have to contain a top-bottom crossing path of $\left(\mathbb{Z}\times P_p\right)^{\star}$.  After completing her first turn, Maker follows the horizontal player $\mathcal{H}$'s winning strategy for the $q$-double-response game on $\mathbb{Z}\times P_p$ (a strategy whose existence is given by Proposition~\ref{proposition: double-response} and our assumption that $p \geqslant 2q \geqslant q+1$). This prevents Breaker from ever creating a top-bottom dual crossing path of $\mathbb{Z}\times P_p$, and, as such, prevents Breaker from ever winning the $(p,q)$-percolation game on $\mathbb{Z}^{2}$. Maker thus has a winning strategy in this case.
\end{proof}

To prove Theorem \ref{theorem: (p,2p)-game on integer lattice}, we shall need a result about a variant of a specific instance of the \textit{box-game}. The box-game was introduced by Chv{\'a}tal and Erd{\H o}s in~\cite{ChvatalErdos78}, and solved in full generality by Hamidoune and Las Vergnas~ \cite{HamidouneLasVergnas87}.  Our variant of the box-game is played as follows.
\begin{definition}[$(q,M,N)$-double-response box-game]\label{def: box double response game}
In the $(q,M,N)$-double-response box-game, two players, BoxBreaker and BoxMaker, play in alternating turns on a board consisting of $N$ boxes, each containing $M$ items. BoxBreaker plays first, and on each of his turns he picks an integer $r\in [q]$ and removes  $r$  boxes from the board. On each of her turns, BoxMaker claims up to a total of $2r$ items from the boxes remaining on the board (where $r$ was the number of boxes removed by BoxBreaker in the preceding turn). BoxMaker wins the game if she manages to claim all $M$ items from some box before BoxBreaker removes it.
\end{definition}
\begin{lemma}\label{lemma: q-double-response-box-game}
If $N \geqslant 4(q+2)^{M}$, then BoxMaker has a winning strategy for the $(q,M,N)$-double-response box-game.
\end{lemma}
\begin{proof}
BoxMaker's winning strategy proceeds in $M+1$ phases.  For every integer $k \in [0,M]$, her strategy (described below) will guarantee that when Phase $k$ begins the following hold:
\begin{enumerate}[(i)]
	\item it is BoxBreaker's turn to play
	\item there are at least $4(q+2)^{M-k}$ boxes on the board from which BoxMaker has claimed $k$ items.
\end{enumerate}
In particular, the above entails that at the beginning of Phase $M$ there is at least one box on the board from which BoxMaker has claimed all $M$ items, so this is a winning strategy for BoxMaker.

We now show by induction on $k$ that BoxMaker can guarantee (i) and (ii) for every $k \in \{0,1, \ldots M\}$. The assumption on $N$ and the definition of the $(q,M,N)$-double-response box-game give us that both (i) and (ii)  hold at the beginning of Phase $0$. Suppose now that we have reached the beginning of Phase $k$ for some $k< M$, and that (i) and (ii) are both satisfied. On each of his turns throughout Phase $k$, BoxBreaker chooses some $r\in[q]$ and removes $r$ boxes from the board; BoxMaker's response will then be to claim one item each from $2r$ different boxes from which she has previously claimed $k$ items, unless there are strictly fewer than $2r$ such boxes left, in which case she claims $2r$ arbitrary items and declares Phase $k$ over.

Let $A$ and $B$ denote the number of boxes remaining on the board from which BoxMaker has claimed $k+1$ and $k$ items respectively. Set $S=2A+B$, and consider how $S$ changes throughout Phase $k$. At the start of Phase $k$, our inductive assumption (ii) tells us $A\geq 0$ and $B \geqslant 4(q+2)^{M-k}$, so that $S \geqslant 4(q+2)^{M-k}$. On each of his turns during Phase $k$, BoxBreaker removes $r$ boxes for some $r \in [q]$, thereby reducing $S$ by at most $2r$.  If  $B\geq 2r$, then BoxMaker's response ensures that $S$ increases back by $2r$. If on the other hand $B<2r$, then BoxMaker first chooses her items arbitrarily, in which case $S$ potentially does not change, and then declares Phase $k$ to be over.  This implies that at the end of BoxMaker's last turn of Phase $k$ we have $S \geqslant 4(q+2)^{M-k}-2q$ and $B < 2q$.  In particular we have
\begin{equation}
A = \frac{S-B}{2} > \frac{4(q+2)^{M-k}-4q}{2} \geqslant 4(q+2)^{M-k-1}, \nonumber
\end{equation}
and it is BoxBreaker's turn to play, so that (i) and (ii) are both satisfied at the beginning of Phase $k+1$, as required.
\end{proof}

%\begin{remark}
%We remark that the condition $N \geqslant 4(q+2)^{M}$ in Lemma \ref{lemma: q-double-response box-game} is not tight, and we have made no attempt to optimise it here.
%\end{remark}

\begin{proof}[Proof of Theorem \ref{theorem: (p,2p)-game on integer lattice}]
Before we get into the technical details, let us give an outline of Breaker's winning strategy. We may assume without loss of generality that $q=2p$ --- any edges Breaker gets to claim above $2p$ can be played arbitrarily without prejudice to Breaker. Recall that the $\ell_{\infty}$-norm of a point $\mathbf{u}=(u_1, u_2)$ in $\mathbb{Z}^2$ is  $\| \mathbf{u}\|_{\infty}=\max\left(\vert u_1\vert, \vert u_2\vert \right)$.

Let us denote the origin of $\mathbb{Z}^2$ by $\mathbf{0}$. Breaker will restrict his attention to the finite subset of the square lattice induced by vertices at $\ell_{\infty}$ distance at most $N(p+1)$ from the $\mathbf{0}$, for some suitably chosen $N$. He partitions this subset into square annuli, where the $k$-th annulus $A(k)$ corresponds to vertices with $\ell_{\infty}$ norm between $k(p+1)$ and $(k+1)(p+1)$. Note that for such annnuli, any path in $\mathbb{Z}^2$ from the inner square to the outer one must contain at least $p+1$ edges.
% In particular, these annuli are each sufficiently wide so that any path from t inside and outside of any one must contain at least $p+1$ edges.

Each of these annuli can be further subdivided into a set of four rectangular strips $R(k)$ and a set $C(k)$ of four `corners'. Breaker will play double-response games with Maker on each of the $R(k)$ and on the union of the $C(k)$: for every edge that  Maker claims from $R(k)$ in her turn, Breaker will claim two suitably chosen edges from $R(k)$ in his response; and for every edge that Maker claims from a corner in her turn, Breaker will claim two edges from suitably chosen corners in his response.

By Proposition~\ref{proposition: double-response}, Breaker has a strategy in (a variant of) the $p$-double-response game on $R(k)$ for preventing Maker from claiming the edges of a path crossing a strip in $R(k)$ from the inner square of the annulus $A(k)$ to its outer square. Furthermore, our choice of $N$ will ensure that by Lemma~\ref{lemma: q-double-response-box-game} Breaker has a strategy for (a variant of) the $p$-double-response box-game played on the collection of corners $C(0), C(1), \ldots C(N-1)$ that guarantees he is able to claim all the edges of some $C(k_0)$.  (The crucial property of the $C(k)$ for this argument to work is that, unlike the $R(k)$, the sizes of the $C(k)$ is constant.) By following these double-response strategies, Breaker will be able to ensure  the component of $\mathbf{0}$ is eventually wholly contained inside some finite set $\bigcup_{k\leq k_0} A(k_0)$, $k_0 \leq N-1$, and hence that he wins the percolation game.

We now fill in the details. Set $N= 4(p+2)^{8p}$. As discussed above, for $k\in\{0, 1,\ldots N-1\}$, let $A(k)$ denote the annulus $A(k)=\{\mathbf{u}:\  \|\mathbf{u}\|_{\infty} \in[k(p+1), (k+1)(p+1)] \}$. Inside $A(k)$, let $R'_1(k)$ denote the subgraph of $\mathbb{Z}^2$ induced by the vertices 
\begin{equation}
\{(x,y):\   \vert x\vert \leq k(p+1), \ k(p+1)\leq y \leq (k+1)(p+1) \}. \nonumber
\end{equation}
We then let $R_{1}(k)$ denote the subgraph of $R'_1(k)$ obtained by removing any edge of $R'_1(k)$ that lies inside $A(k+1)$ or $A(k-1)$. We further let $R_2(k)$, $R_3(k)$ and $R_4(k)$ be the subgraphs obtained by rotating $R_1(k)$ around the origin by  angles of $\pi/2$, $\pi$ and $3\pi/2$ respectively. Thus $R_1(k)$, $R_2(k)$, $R_3(k)$ and $R_4(k)$ are respectively part of the topmost, leftmost, bottommost and rightmost rectangular strips in the square annulus $A(k)$. Let $R(k)$ denote the union of the $R_i(k)$, $i\in [4]$. Denote by $L'(k)$ the collection of all edges which are from a vertex in $R(k)$ to a vertex in $A(k)\setminus V(R(k))$. We then let $L(k)$ denote the subset of $L'(k)$ obtained by removing any edge that lies in $A(k-1)$ or $A(k+1)$.  The set of dual edges corresponding to $L(k)$ consist of four L-shapes, each containing $2p$ edges. Let us write $R^{\star}_{i}(k)$ (respectively $L^{\star}(k)$) for the set of edges dual to $R_{i}(k)$ (respectively $L(k)$).  These four strips $R^{\star}_i(k)$, $i\in[4]$, may be viewed as (rotated, translated) subgraphs of $\mathbb{Z}\times P_{p+1}$.  See Figure~\ref{Fig2} for an example of these sets. 
\begin{figure}[ht]
	\centering
\includegraphics[scale=1]{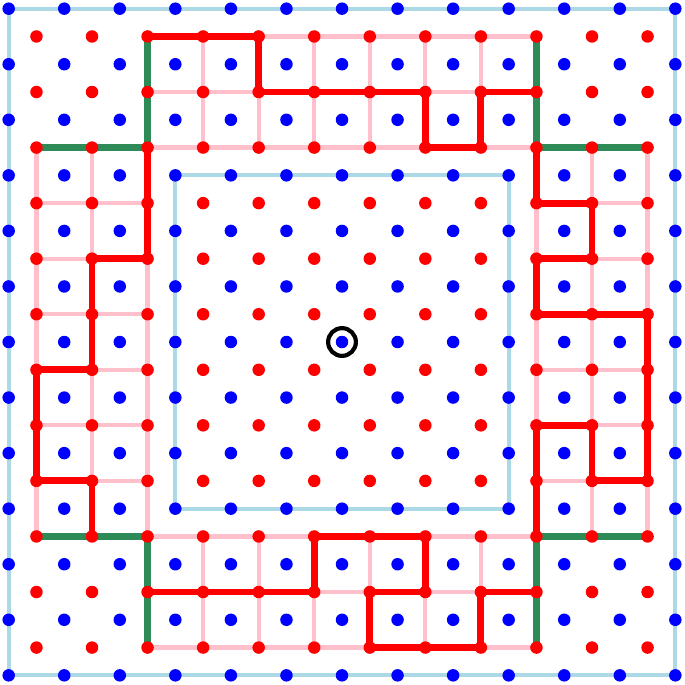}
	\caption{This figure shows some of the sets involved in our winning Breaker strategy in the special case $p = 2$. The blue vertices above represent the vertices of $\mathbb{Z}^2$, while the red vertices represent the dual vertices from $\left(\mathbb{Z}^2 \right)^{\star}=\mathbb{Z}^2+(0.5, 0.5)$. We see in the figure the boundaries of the annulus $A(1)$ (in light blue), the edge-set $L^{\star}(1)$ (in green), and the sets $R^{\star}_{1}(1)$, $R^{\star}_{2}(1)$, $R^{\star}_{3}(1)$ and $R^{\star}_{4}(1)$ (in pink), as well as crossing paths $\pi_i$ (in red) for each of the $R^{\star}_{i}(1)$.   The circled vertex is the origin of $\mathbb{Z}^{2}$. Note that the union of the red dual edges with the green dual edges contains a dual cycle around the origin; claiming the corresponding edges constitutes a win for Breaker in the percolation game.}
	\label{Fig2}
\end{figure}

We can now formally describe Breaker's winning strategy. Set $L=\bigcup_{k\leq N-1}L(k)$. %Call a path from the inner square of $A(k)$ to the outer square contained entirely within some $R_i(k)$ a \emph{breakout path}. 
In any given turn, Maker claims $p_{L}$ edges from $L$ and $p_{i,k}$ edges from $R_i(k)$ for each $i\in[4]$ and each $k\in\{0,1,\ldots N-1\}$, where $p_{L}+\sum_{i,k}p_{i,k}\leq p$. Breaker treats $L$ as an instance of the box-game, where the boxes are the $L(k)$, each of which contains exactly $8p$ elements. He responds to Maker's moves on $L$ by claiming $2p_L$ edges from $L$ according to the winning strategy for BoxMaker in the $(p, 8p, N)$-double-response box-game given in Lemma~\ref{lemma: q-double-response-box-game} (which is possible since we chose $N\geq 4(p+2)^{8p}$). Furthermore, for each $i,k$, Breaker treats $R^{\star}_i(k)$ as a (rotated, translated and truncated) instance of the $p$-double-response game on $\mathbb{Z}\times P_{p+1}$; he responds to Maker's moves on $R_i(k)$ by claiming $2p_{i,k}$ edges from $R^{\star}_i(k)$ according to the winning strategy for the horizontal player $\mathcal{H}$ given by Proposition~\ref{proposition: double-response}. This uses up $2p_{L}+\sum_{i,k}2p_{i,k}\leq 2p$ of the edges Breaker is allowed to claim. If he has any edges left over (for example if Maker played fewer than $p$ edges inside $\bigcup_{k\leq N-1}A(k)$ in her preceding turn), then Breaker claims some arbitrary edges inside $S_p:=\bigcup_{k\leq N-1}A(k)$. This ensures that between successive turns for Breaker, at least $2p$ edges of $S_p$ are claimed by the players. Thus within $T \leq 4(2N+1)^2/2p$ turns, all edges inside $S_P$ have been claimed by one of the players.

Since Breaker's strategy on $L$ is a winning one we have that by turn $T$ of the game, there exists some $k_0\in\{0,1,\ldots N-1\}$ such that Breaker has claimed all of the edges from $L(k_0)$.  Since Breaker's strategy on each of the $R^{\star}_i(k_0)$ is a winning one, by turn $T$ Maker has failed for every $i\in [4]$ to claim the edges of a path in $R_i(k_0)$ from the inner square of $A(k_0)$ to its outer square. By a standard result on planar duality (see e.g.~\cite[Chapter 3, Lemma 1]{BollobasRiordan06}), this implies Breaker has claimed edges corresponding to dual paths $\pi_i$, $i\in [4]$, where $\pi_1$ and $\pi_3$ are left-right crossing paths of $R^{\star}_1(k_0)$ and $R^{\star}_3(k_0)$ and $\pi_2$ and $\pi_4$ are top-bottom crossing paths of $R^{\star}_2(k_0)$ and $R^{\star}_4(k_0)$ respectively. Now, the union of the $\pi_i$ together with $L^{\star}(k_0)$ contains a dual cycle $\sigma$ inside $A(k_0)$ surrounding the origin (see Figure \ref{Fig2} for an example in the case $p = 2$, $k_0=1$).
By turn $T$ Breaker has thus claimed edges whose dual is such a cycle $\sigma$. This implies that by that turn, the origin is contained in a component strictly contained inside the finite square $\bigcup_{k\leq k_0}A(k)\subseteq S_p$. In particular, Breaker has won the $(p,2p)$-percolation game.
\end{proof}

\section{Concluding remarks and open questions}\label{section: concluding remarks and open questions}
There are many questions arising from our work. Given our motivation for looking at Maker-Breaker games in the context of percolation, the one we would most like answered is the question posed in the introduction and restated below, which asks about a Maker-Breaker analogue of the Harris--Kesten theorem:
\begin{question}\label{question: HK for Maker-Breaker}
	Consider the $(p,q)$-percolation game on $\mathbb{Z}^2$. Does there exist a \emph{critical bias} $b_{\star}> 0$ such that for any $\varepsilon>0$ and all $p$ sufficiently large,  $(b_{\star}+\varepsilon)p<q$ implies Breaker has a winning strategy, while $(b_{\star}-\varepsilon) p>q$ implies Maker has a winning strategy?
\end{question}

% does there exist a \emph{critical bias} $b_{\star}$ for the $(p,q)$-percolation game on the square integer lattice $\mathbb{Z}^2$ such that $b_{\star}p<q$ implies Breaker has a winning strategy while $b_{\star}p>q$ implies Maker has a winning strategy?

More generally, for fixed $p$ it is natural to ask for the exact threshold value $q$ at which the $(p,q)$-percolation game goes from being a Maker win to a Breaker win under optimal play.
\begin{question}\label{question: q-threshold}
	Let $p \geqslant 2$ be a fixed integer. What is the least integer $q$ such that Breaker has winning strategy for the $(p,q)$-percolation game on the square integer lattice $\mathbb{Z}^2$?
\end{question}
Theorem~\ref{theorem: (p,2p)-game on integer lattice} shows this threshold is at most $2p$, while Theorem~\ref{theorem: (2q,q)-game on integer lattice} shows it is at least $p/2$. One natural way of investigating Question~\ref{question: q-threshold} would be to settle who has a winning strategy in the $(p,q)$-percolation game for small values of $p$ and $q$. The smallest open case is when $p=q=2$. We believe this $(2,2)$-case already contains many of the difficulties inherent to the general case.
\begin{question}\label{question: (2,2)-case}
	Which of Maker and Breaker has a winning strategy for the $(2,2)$-percolation game on $\mathbb{Z}^2$?
\end{question}
\noindent Another appealing family of special cases are $(p,q)=(k, 2k-1)$ and $(p,q)=(2k-1,k)$ for $k\in \mathbb{Z}_{\geq 2}$. Does having almost but not quite twice the power of your opponent guarantee you have a winning strategy? In particular, we have the following questions:
\begin{question}\label{question: (p, 2p-1), (2q-1,q)}
	\begin{enumerate}[(i)]
		\item Which of Maker and Breaker has a winning strategy for the $(3,2)$-percolation game on $\mathbb{Z}^2$?
		\item Which of Maker and Breaker has a winning strategy for the $(2,3)$-percolation game on $\mathbb{Z}^2$?
	\end{enumerate}
\end{question} 

In a more theoretical direction, one could ask about the existence of a critical bias in a more general setting (possibly ignoring some finite number of `bad' pairs $(p,q)$).
\begin{question}\label{question: critical bias}
Let $\Lambda$ be a vertex-transitive infinite, locally finite, connected graph. Does there exist a critical bias $b_{\star}=b_{\star}(\Lambda)\geq 0$ for the $(p,q)$-percolation game on $\Lambda$, in the sense that for all $\varepsilon>0$ there exists $p_0\in \mathbb{N}$ such that for all $p\geq p_0$, $(b_{\star}+\varepsilon)p < q$ implies Breaker has a winning strategy  while $(b_{\star}-\varepsilon)p > q$ implies Maker has a winning strategy?
\end{question}
To answer such a question, it would be good to understand what, if anything, the existence of a winning strategy for one of the players in the $(p,q)$-percolation game says about the $(mp, mq)$-percolation game, and vice-versa, where $m\in \mathbb{N}$. We observe that if any relationship between the two games does exist, it is not entirely trivial--- note for instance that, as we showed, Maker wins the $(1,1)$-percolation game on $\mathbb{Z}^2$ even if she forfeits her first turn, while clearly Breaker will win the $(4,4)$-percolation game on the same graph if he is allowed to play first.

%\begin{remark}\label{remark: zero-one law}
Another interesting theoretical question is whether Breaker can afford to skip his turn early on if he has a winning strategy. Indeed, observe that our proof of Theorem~\ref{theorem: bi-regular tree}, establishes that if $p(b-2)-\lceil\frac{p}{a}\rceil(b-a)<q$, then Breaker has a winning strategy in a very strong sense: he wins even if Maker is allowed to pre-emptively claim some arbitrary finite number of edges before the start of the $(p,q)$-percolation game on $T$. This is somewhat reminiscent of the phenomenon in percolation theory whereby the existence of an infinite connected component is independent of what happens in any finite region (this is a special case of the celebrated Kolmogorov $0$-$1$ law). It is very  natural to ask whether a similar phenomenon might also occur for more general percolation games.
%\end{remark}
\begin{question}\label{question: zero-one law}
	Let $\Lambda$ be a connected unrooted infinite graph.  Allow Maker to select a vertex $v_0$ of $\Lambda$ to be the root, and then play the $(p,q)$-percolation game on $\Lambda$ as normal. Suppose Breaker has a winning strategy. Is it true that Breaker still has a winning strategy even if Maker is allowed to pre-emptively claim some finite number of edges before the game starts?
\end{question}

In this paper, we focussed on Bethe lattices and integer lattices, though we gave a definition of our percolation game which is valid in a much more general context. It is natural to ask what happens on other lattices, such as the triangular lattice. We also focussed on games where the players take turns claiming edges in the graph (which corresponds to \emph{bond percolation} in percolation theory); it would be interesting to study the variant where they take turns claiming vertices in the graph instead (which would be the analogue of \emph{site percolation}). Observe this is a more general class of games, since claiming edges in a graph $\Lambda$ is equivalent to claiming vertices in the line graph of $\Lambda$. Of particular interest would be percolation game analogues of the celebrated game of Hex, which corresponds to site percolation on the triangular lattice.

Our proof of Theorem~\ref{theorem: (p,2p)-game on integer lattice} shows that in the $(p,2p)$-percolation game on $\mathbb{Z}^2$, Breaker can ensure the origin is contained in a connected component of order $p^{16p +O(1)}$ by the end of the game. It is natural to ask whether in fact Breaker can do a lot better: can he ensure, for instance, that the origin is contained in a component of order subexponential in $p$?
\begin{question}
	Set $M_p$ to be the least integer such that Breaker has a winning strategy in the $(p, 2p)$-percolation game on $\mathbb{Z}^2$ ensuring the origin is contained in a component of order at most $M_p$. What is the asymptotic behaviour of $M_p$ as $p\rightarrow \infty$?
\end{question} 
\noindent Our analysis in the proof of Theorem~\ref{theorem: (p,2p)-game on integer lattice} is certainly wasteful, so we cannot hazard a guess as to what the right answer to this question might be.

Another open problem is to extend our results on trees. We determined which of Maker and Breaker has a winning strategy for the $(p,q)$-percolation game on regular and bi-regular trees. Is there a general criterion for determining who has a winning strategy on arbitrary infinite trees? One difficulty here may be that global invariants such as the branching number may ignore some local bottlenecks, which don't affect the asymptotic growth of the tree but do allow Breaker to overwhelm Maker and cut off the root. One way to address this might be to play the percolation game on non-rooted trees and allow Maker in her first turn to choose the location of the root she has to protect.

Finally, given our motivation for looking into percolation games, it would be natural to ask what happens when they are played on random boards: given a rooted, connected infinite graph $\Lambda$, let $\Lambda_{\theta}$ be a $\theta$-random subgraph obtained by including each edge of $\Lambda$ in $\Lambda_{\theta}$ with probability $\theta$, independently of all other edges. For suitable values of $\theta\in [0,1]$, the subgraph $\Lambda_{\theta}$ will \emph{almost surely} (a.s.) contain an infinite connected component. One can then play a percolation game on $\Lambda_{\theta}$ by allowing Maker in her first turn to choose the location of a root, and then playing the $(p,q)$-percolation game as normal. In particular, given Maker wins the $(1,1)$-percolation on $\mathbb{Z}^2$, one could ask how much one can `thin' the board (i.e. how small a $\theta$ we can take) while ensuring she a.s. retains her advantage in the $(1,1)$-percolation game over Maker.
\begin{question}
	Set
	\[\theta_{\star}=\inf\left\{\theta\in \bigl(\frac{1}{2},1\bigr]: \  \textrm{Maker a.s. has a winning strategy for the $(1,1)$-percolation game on $\left(\mathbb{Z}^2\right)_{\theta}$} \right\}.\]
	What is the value of ${\theta}_{\star}$ ?
\end{question}
\subsection*{Acknowledgements}
We are very grateful to the two anonymous referees whose scholarship and careful work greatly enhanced the correctness and presentation of the paper.

This research was made possible thanks to the support of the Swedish Research Council grant VR 2016-03488, which we gratefully acknowledge.

\end{document}